\documentclass[11pt]{article}

\usepackage{amsmath,amsthm,amsfonts, amssymb, latexsym,stmaryrd}

\begin{document}

\newtheorem*{theo}{Theorem}
\newtheorem*{pro} {Proposition}
\newtheorem*{cor} {Corollary}
\newtheorem*{lem} {Lemma}
\newtheorem{theorem}{Theorem}[section]
\newtheorem{corollary}[theorem]{Corollary}
\newtheorem{lemma}[theorem]{Lemma}
\newtheorem{proposition}[theorem]{Proposition}
\newtheorem{conjecture}[theorem]{Conjecture}

\theoremstyle{definition}
 \newtheorem{definition}[theorem]{Definition}
  \newtheorem{example}[theorem]{Example}
   \newtheorem{remark}[theorem]{Remark}
   
\newcommand{\Naturali}{{\mathbb{N}}}
\newcommand{\Reali}{{\mathbb{R}}}
\newcommand{\Complessi}{{\mathbb{C}}}
\newcommand{\Toro}{{\mathbb{T}}}
\newcommand{\Relativi}{{\mathbb{Z}}}
\newcommand{\HH}{\mathfrak H}
\newcommand{\KK}{\mathfrak K}
\newcommand{\LL}{\mathfrak L}
\newcommand{\as}{\ast_{\sigma}}
\newcommand{\tn}{\vert\hspace{-.3mm}\vert\hspace{-.3mm}\vert}
\def\A{{\cal A}}
\def\B{{\cal B}}
\def\E{{\cal E}}
\def\F{{\cal F}}
\def\H{{\cal H}}
\def\J{{\cal J}}
\def\K{{\cal K}}
\def\L{{\cal L}}
\def\N{{\cal N}}
\def\M{{\cal M}}
\def\MI{{\cal MI}}
\def\gM{{\frak M}}
\def\O{{\cal O}}
\def\P{{\cal P}}
\def\S{{\cal S}}
\def\T{{\cal T}}
\def\U{{\cal U}}
\def\V{{\mathcal V}}
\def\qed{\hfill$\square$}

\title{On maximal ideals in certain reduced twisted
C*-crossed products}

\author{Erik B\'edos,
Roberto Conti \\}
\date{\today}
\maketitle
\markboth{R. Conti, Erik B\'edos}{
}
\renewcommand{\sectionmark}[1]{}
\begin{abstract} We consider a twisted action of a discrete group $G$ on a unital C$^*$-algebra $A$ and give conditions ensuring that there is a bijective correspondence between the maximal invariant ideals of $A$ and  the maximal ideals in  the associated reduced C*-crossed product. 


\end{abstract}

\section{Introduction} 

\medskip  Let $G$ denote a discrete group. We recall that $G$ is called C$^*$-simple \cite{Bed, dH} whenever the reduced group C$^*$-algebra $C_r^*(G)$ is simple. The class of C$^*$-simple groups is vast and many subclasses are known. We refer to \cite{dH} for a nice overview on this topic. Some related articles that have appeared recently are \cite{Poz, PeTh, dHP, Iva, TD, OO}. 

One open problem concerning C$^*$-simple groups is the following (cf.\  \cite{dHS}). Assume that a C$^*$-simple group $G$ acts on a unital C$^*$-algebra $A$ in a minimal way (that is, the only invariant  ideals\footnote{By an ideal in a C$^*$-algebra, we always mean a closed two-sided ideal, unless otherwise specified.} of $A$ are  $\{0\}$ and $A$). Is the associated reduced C$^*$-crossed product simple ? This question was  answered positively by P.\ de la Harpe and G.\ Skandalis \cite{dHS} when $G$ is a Powers group (e.g.\ when $G$ is a free nonabelian group, as in Powers' original work \cite{Pow}). Their result was later extended to cover weak Powers groups and twisted actions (see \cite{BN, Bed}). Moreover, this question was also answered positively by M.\ Bekka, M.\ Cowling and P.\ de la Harpe  \cite{BCdH} in the case where $G$ has the so-called property ($P_{\rm com}$), e.g. $G=PSL(n, \Relativi)$ for $n\geq 2$.  (We will see in this paper that their argument may be adapted to  handle also the  case of twisted actions.)  

As remarked by de la Harpe and Skandalis in \cite{dHS}, if  $A$ is not assumed to be unital, then there are examples where the above question has a negative answer. They also give an example of an action of a Powers group on a unital C$^*$-algebra $A$ such that $A$ has exactly one nontrivial invariant ideal while the associated reduced C$^*$-crossed product has infinitely many ideals.  This could be taken as an indication that it is not possible  to say something of interest about the lattice of ideals in a reduced C$^*$-crossed product involving a non minimal action of a Powers group. Our main purpose in this paper is to show that for a large class $\P$ of C$^*$-simple groups, containing all weak Powers groups, one may in fact describe the maximal ideals of the reduced crossed product in the case of an exact twisted action on a unital C$^*$-algebra by a group in $\P$  (see Corollary \ref{main}). The class $\P$ consists of all PH groups (as defined by S.D.\ Promislow \cite{Pro}) and all groups with property ($P_{\rm com}$). In the case of a weak Powers group, this result was briefly discussed in \cite[Example 6.6]{BeCo2}. 

As a part of our work, we introduce a certain property for a twisted unital discrete C$^*$-dynamical system $\Sigma= (A, G, \alpha, \sigma)$ that we call (DP) (named after Dixmier and Powers). This property, which is
weaker than the Dixmier property for the reduced crossed product $C_r^*(\Sigma)$, is always satisfied by $\Sigma$ whenever $G$ belongs to the class $\P$ (see Theorem \ref{P} and Section \ref{appendix}). Moreover, we prove that if $\Sigma$ is exact \cite{S, BeCo2} and has property (DP), then there is a one-to-one correspondence between the set of maximal ideals of $C_r^*(\Sigma)$ and the set of maximal invariant ideals of $A$, and also  a one-to-one correspondence between the set of all tracial states of  $C_r^*(\Sigma)$ and the set of invariant tracial states of $A$ (see Theorem \ref{max} and Proposition \ref{tracial}).

To illustrate the usefulness of our results, we describe in Section \ref{exa} the maximal ideal space of some  C$^*$-algebras that may be written as $C_r^*(\Sigma)$ for a suitable $\Sigma$. These examples include   $C_r^*(\Gamma)$ for any discrete group $\Gamma$ such that the quotient of $\Gamma$ by its center is exact and belongs to $\P$,  $C_r^*\big(\Relativi^3\rtimes SL(3, \Relativi)\big)$, and the "twisted" Roe algebras $C_r^*(\ell^\infty(G), G, {\rm lt}, \sigma)$ when $G$ is exact, belongs to $\P$ and the $2$-cocycle $\sigma$ is scalar-valued. 

We use standard notation. For instance, if $A$ is a unital C$^*$-algebra, then $\U(A)$ denotes the unitary group of $A$ and ${\rm Aut}(A)$ the group of all $^*$-automorphisms of $A$. If $\H$ is a Hilbert space, then $\B(\H)$ denotes the bounded linear operators on $\H$.

\section{Preliminaries}
Throughout this paper, we let  $\Sigma = (A, G, \alpha,\sigma)$  denote a twisted, 
unital, discrete 
$C^*$-dynamical system (see for instance \cite{ZM, PaRa, Wi, Ec}). Thus,
$A$ is a $C^*$-algebra with unit $1$, 
$G$ is a discrete group with identity $e$
and $(\alpha,\sigma)$ is a  twisted
action of $G$ on $A$, that is,
$\alpha$ is a map from $G$ into ${\rm Aut}(A)$  
and  $\sigma$ is a map from $G \times G$ into $\, \U(A)$,
satisfying
\begin{align*}
\alpha_g \circ \alpha_h & = {\rm Ad}(\sigma(g, h)) \circ  \alpha_{gh} \\
\sigma(g,h) \sigma(gh,k) & = \alpha_g(\sigma(h,k)) \sigma(g,hk) \\
\sigma(g,e) & = \sigma(e,g) = 1 \ 
\end{align*}
for all $g,h,k \in G$. Of course,  ${\rm Ad}(v)$ denotes here the (inner) automorphism  of $A$ implemented by some $v \in \U(A)$.
One deduces easily that
$$\alpha_e = {\rm id}, \ \sigma(g,g^{-1}) = \alpha_g(\sigma(g^{-1},g))$$
and
$$\alpha^{-1}_g
= \alpha_{g^{-1}} \circ  {\rm Ad}(\sigma(g,g^{-1})^*)
= {\rm Ad}(\sigma(g^{-1},g)^*) \circ  \alpha_{g^{-1}}\,.$$

\medskip
Note that if  $\sigma$ is 
trivial, that is, $\sigma(g,h)=1$ for all $g,h \in G$, then $\Sigma$ is an ordinary $C^*$-dynamical system.

The reduced crossed product $C^*_r(\Sigma)$ associated with $\Sigma$ may (up to isomorphism)  be characterized as follows \cite{ZM, BeCo1}: 

\begin{itemize}
\item $C_r^*(\Sigma)$ is generated (as a C$^*$-algebra) by (a copy of) $A$  and a family $\{\lambda(g)\, | \,  g\in G\}$ of unitaries satisfying $$\alpha_g(a) = \lambda(g)\, a \, \lambda(g)^*\, \, \text{and}\,\, \lambda(g)\,\lambda(h) = \sigma(g,h)\, \lambda(gh)$$ for all $g,h \in G$ and $a \in A$,

\item there exists a faithful conditional expectation $E : C_r^*(\Sigma) \to A$ such that $E(\lambda(g)) =0$ for  all $g \in G\,,\,  g \neq e\,.$
\end{itemize}

\noindent One easily cheks that  the expectation $E$ is equivariant, that is, we have
 $$E\big(\lambda(g)\,x\,\lambda(g)^*\big) = \alpha_g\big(E(x)\big)$$
 for all $g \in G, \, x\in C_r^*(\Sigma)$. As is well known, it follows that if $\varphi$ is a tracial state on $A$ which is invariant (i.e.\ $\varphi(\alpha_g(a)) = \varphi(a)$ for all $g\in G, \, a\in A$), then $\varphi\circ E$ is a tracial state on $C_r^*(\Sigma)$ extending $\varphi$.

 \medskip Let $J$ denote an invariant ideal of $A$ and set $\Sigma/J = (A/J, G, \dot{\alpha},\dot{\sigma})$, where $(\dot{\alpha},\dot{\sigma})$ denotes the  twisted action of $G$ on $A/J$ naturally associated with $(\alpha, \sigma)$. 
 
 We will let $\langle J\rangle$ denote the ideal of $C_r^*(\Sigma) $ generated by $J$. Any ideal of this form is called an {\it induced ideal} of  $C_r^*(\Sigma) $. Moreover, we  will let $\tilde{J}$ denote  the kernel of the canonical $^*$-homomorphism from $C_r^*(\Sigma)$ onto $C_r^*(\Sigma/J)$. It is elementary to check that we have $E\big(\langle J\rangle) = J $ and $\langle J\rangle \subset \tilde{J}$.  Another useful fact is that
  $$ \tilde{J}= \big\{ x \in C_r^*(\Sigma)\mid \widehat{x}(g) \in J \, \, \text{for all}\, \, g \in G\big\}\,$$
 where  $\widehat{x}\,(g) = E\big(x\,\lambda(g)^*\big) \,$ for each $  x \in C_r^*(\Sigma), \, g\in G\,$.
  This may for instance be deduced from the proof of \cite[Theorem 5.1]{Ex3} by considering  $C_r^*(\Sigma)$ as topologically graded C$^*$-algebra over $G$: 
  $$C_r^*(\Sigma) =\overline{\bigoplus_{g\in G} A_g}^{\, \|\cdot\|}\, ,$$
  where  $A_g =  \big\{a\, \lambda(g)\, |Ê\, a \in A\big\} $ for each $g \in G$.

Following \cite{S, BeCo2}, we will say that the system $\Sigma$ is {\it exact}  whenever we have $\langle J\rangle = \tilde{J}$ for every invariant ideal $J$ of $A$. It is known \cite{Ex2} that $\Sigma$ is exact whenever $G$ is exact. It is also known \cite{BeCo2} that $\Sigma$ is exact whenever there exists a Fourier summing net for $\Sigma$ preserving the invariant ideals of $A$. This latter condition is for instance satisfied when $\Sigma$ has Exel's approximation property \cite{Ex}, e.g.\ when the associated action of $G$ on  the center  $Z(A)$ of $A$, obtained by restricting $\alpha$ to $Z(A)$, is amenable (as being defined in \cite{BrOz}). 

We include here two lemmas illustrating the impact of the exactness of $\Sigma$ on the lattice of ideals of $C_r^*(\Sigma)$.
\begin{lemma} \label{basic} Let $\J$ be an ideal of $C_r^*(\Sigma)$ and set $J=\overline{E(\J)}$. Then $J$ is an invariant  ideal of $A$ such that $\J \subset \tilde{J}$.  Hence, if $\Sigma$ is exact, we have $\J\subset \langle J \rangle$.
\end{lemma}
\begin{proof} As $E$ is a conditional expectation, it follows readily that $J$ is an ideal of $A$. The invariance of $J$ is an immediate consequence of the equivariance of $E$. Let now $x\in \J$. Then, for each $g \in G$, we have
$x\, \lambda(g)^* \in \J$, so $$\widehat{x}(g) = E\big(x\, \lambda(g)^*\big) \in E(\J) \subset J\,.$$
Hence, $x\in \tilde{J}$. This shows that $\J \subset \tilde{J}$. The last assertion follows then from the definition of exactness.

\end{proof}
An  ideal $\J$ of $C_r^*(\Sigma)$ is called $E${\it-invariant} if $E(\J) \subset \J$. Equivalently, $\J$ is $E$-invariant  whenever  $E(\J) = \J \cap A$ (so $E(\J)$ is necessarily closed in this case). Any induced ideal of  $C_r^*(\Sigma)$ is easily seen to be $E$-invariant. The converse is true if $\Sigma$ is exact, as shown below. (When $G$ is exact, this is shown in \cite{Ex3}; see \cite{BeCo2} for the case where there exists a Fourier summing net for $\Sigma$ preserving the invariant ideals of $A$). 

\begin{lemma} \label{E-inv} Let $\J$ be an $E$-invariant ideal of $C_r^*(\Sigma)$. If $\Sigma$ is exact, then $\J$ is an induced ideal. Indeed, we have $\J = \langle E(\J) \rangle$ in this case.  
\end{lemma}
\begin{proof} Note that since $E(\J) = \J \cap A$ is closed, it is an invariant ideal of $A$ (cf.\ Lemma \ref{basic}). Assume that $\Sigma$ is exact. Then Lemma \ref{basic} gives that $\J \subset \langle E(\J) \rangle$. On the other hand, since $E(\J) \subset \J$, we have $\langle E(\J) \rangle \subset \J$. Hence, $\J = \langle E(\J) \rangle$, as asserted.

\end{proof}
\section{On maximal ideals and reduced twisted C$^*$-crossed products}

We set $\U_\Sigma =\U\big(C_r^*(\Sigma)\big)$.  When $S$ is a subset of a (complex) vector space, we let co$(S)$ denote the convex hull of $S$.

\begin{definition} The system $\Sigma$ is said to have {\it property $(DP)$} whenever we have
\begin{equation}\label{Conv}  0 \, \in \, \overline{\text{co}\{v\, y \, v^*\, |Ê\, v \in \U_\Sigma\}}^{\,\, \|\cdot\|}
\end{equation}
for every $y\in C_r^*(\Sigma)$ satisfying $ \, y^*=y$ and $ \, E(y) =0$.

\end{definition} 

\begin{remark} \label{strongDP} 
Let $\U_G$ be the subgroup of $\U_\Sigma$ generated by the $\lambda(g)$'s. The above definition might be strengthened by replacing $\U_\Sigma$ with $\U_G$, that is, by requiring that 
\begin{equation}\label{Conv2}  0 \, \in \, \overline{\text{co}\{v\, y \, v^*\, |Ê\, v \in \U_G\}}^{\,\, \|\cdot\|}
\end{equation}
for every $y\in C_r^*(\Sigma)$ satisfying $ \, y^*=y$ and $ \, E(y) =0$.
 All the examples of systems we are going to describe satisfy this strong form of property (DP). It can be shown (see Proposition \ref{strDP}) that if $\Sigma$ has this strong property (DP), then (\ref{Conv2}) holds for every $y\in C_r^*(\Sigma)$ satisfying  $ \, E(y) =0$. It is not clear to us that if $\Sigma$ has property (DP), then (\ref{Conv}) holds for every such $y$. 
\end{remark}

\begin{remark}  
 We recall that a unital C$^*$-algebra $B$  is said to have the {\it Dixmier property} if 
 $$\overline{\text{co}\{u\, b \, u^*\, |Ê\, u \in \U(B)\}}^{\,\, \|\cdot\|}\, \cap \, \Complessi\cdot 1 \, \neq \emptyset $$
 for every $b \in B$. As shown by L.\ Zsido and U.\ Haagerup  in \cite{HZ},  
 $B$ is simple with at most one tracial state if and only if $B$ has the Dixmier property. 
 Using \cite[Corollaire, p.\ 175]{HZ}, it follows  that if $C^*_r(\Sigma)$ has the Dixmier property, then $\Sigma$ has the property (DP) 
 introduced above. Property (DP) may be seen as a kind of relative Dixmier property for the pair $(A, C_r^*(\Sigma))$, generalizing the property considered by R.\ Powers \cite{Pow} in the case where $\Sigma=(\Complessi, \mathbb{F}_2, \rm{id}, 1)$. It should not be confused with the notion of relative Dixmier property for inclusions of C$^*$-algebras considered  
 by S.\ Popa 
 in \cite{Po}. 
 \end{remark}

\medskip A first consequence of property ($DP$) is the following:

\begin{proposition} \label{tracial} Assume $\Sigma$ has property $(DP)$. Then the map $\varphi \to \varphi\circ E$ is a bijection between the set of invariant tracial states of $A$ and the set of tracial states of $C_r^*(\Sigma)$. Especially, $C_r^*(\Sigma)$ has a unique tracial state if and only if $A$ has a unique invariant tracial state. 
\end{proposition}
\begin{proof} It is clear that this map is injective, so let us prove that it is surjective. Let therefore $\tau$ be a tracial state on $C_r^*(\Sigma)$ and let $\varphi$ denote the tracial state of $A$ obtained by restricting $\tau$ to $A$. It follows from the covariance relation that $\varphi$ is invariant. We will show that $\tau = \varphi\circ E$. 

Let $x^*=x \in C_r^*(\Sigma)$ and $\varepsilon >0$. Set $y= x - E(x)$. As $y^* = y$ and $E(y) = E\big(x-E(x)\big)= E(x) - E(x) =0$, property ($DP$) enables us to pick  $v_1, \ldots , v_n \in \U_\Sigma$ and $t_1, \ldots, t_n \in [0,1]$ satisfying $\sum_{i=1}^nt_i =1$ such that 
$$ \Big\| \, \sum_{i=1}^n t_i \, v_i\, y\, v_i^* \, \Big\| \, < \, \varepsilon\,.$$
As $\tau$ is a tracial,  we have $$\tau\Big( \sum_{i=1}^n t_i \, v_i\, y\, v_i^* \Big)=  \sum_{i=1}^n t_i \, \tau( y) = \tau(y)\,,$$ 
so we get 
$$|\tau(y)| = \Big| \tau\Big( \sum_{i=1}^n t_i \, v_i\, y\, v_i^*\Big)\Big| \leq \Big\| \, \sum_{i=1}^n t_i \, v_i\, y\, v_i^* \, \Big\| \, < \, \varepsilon\,.$$
Hence, we can conclude that $\tau(y) =0$. This gives that $$\tau(x) = \tau(E(x)) = (\varphi\circ E)(x)\,.$$ So $\tau$ agrees with $\varphi\circ E$ on the self-adjoint part of $C_r^*(\Sigma)$, and therefore on the whole of $C_r^*(\Sigma)$ by linearity. 

\end{proof}
Next, we have:
\begin{proposition} \label{proper} Assume that $\Sigma$ has property $(DP)$ and let $\J$ be a proper ideal of $C_r^*(\Sigma)$. Set $J=\overline{E(\J)}$. Then $J$ is a proper invariant  ideal of $A$. 
\end{proposition}
\begin{proof} 
We know from Lemma \ref{basic} that $J$ is an invariant ideal of $A$. Assume that $J$ is not proper, i.e.,  $\overline{E(\J)} = A$. Since $A$ is unital, we have $E(\J) = A$. So we may pick $x \in \J$ such that $E(x) = 1$. 

Set $z = x^*x \in \J^+$. Using the Schwarz inequality for complete positive maps \cite{BrOz}, we get
$$E(z) = E(x^* x) \geq E(x)^*E(x) = 1\,.$$
Now, set $y = z -E(z)$, so $y^*= y \in  C_r^*(\Sigma)$ and $E(y) = 0$. Since $\Sigma$ has property $(DP)$, we can find $v_1, \ldots , v_n \in \U_\Sigma$ and $t_1, \ldots, t_n \in [0,1]$ satisfying $\sum_{i=1}^nt_i =1$ such that 
$$(*)\quad \Big\| \, \sum_{i=1}^n t_i \, v_i\, z\, v_i^* \, - \sum_{i=1}^n t_i \, v_i \,E(z)\, v_i^*\, \Big\|  = 
 \Big\| \, \sum_{i=1}^n t_i \, v_i\, y\, v_i^* \, \Big\| \, < \, \frac{1}{2}\,\,.$$
Setting $z'=\sum_{i=1}^n t_i \, v_i\, z\, v_i^*\,$, we have $z' \in \J^+$. Since $E(z) \geq 1$, we also have $$\sum_{i=1}^n t_i \, v_i \,E(z)\, v_i^* \geq 1\,.$$ Hence, it follows from $(*)$ that $z'$ is invertible. So we must have $\J=C_r^*(\Sigma)$, which contradicts the properness of $\J$.  This shows that $J$ is proper. 

\end{proof}

\begin{corollary} \label{simple} Assume $\Sigma$ has property $(DP)$ and is minimal $($that is, $\{0\}$ is the only proper invariant ideal of $A$$)$. Then 
$C_r^*(\Sigma)$ is simple. 
\end{corollary}
\begin{proof} Since $E$ is faithful, this follows immediately from Proposition \ref{proper}.

\end{proof}

\noindent If $\Sigma$  is exact and has property $(DP)$,  
we can in fact characterize the maximal ideals of $C_r^*(\Sigma)$.
We therefore set
\begin{itemize}
\item[] $\MI(A) = \big\{ J \subset A \mid J \, \, \text{is a maximal invariant ideal of} \, \, A\big\}$,  
\item[] $\M\big(C_r^*(\Sigma)\big) = \big\{ \J \subset C_r^*(\Sigma) \mid \J \, \, \text{is a maximal ideal of} \, \, C_r^*(\Sigma)\big\}$.   
\end{itemize}
It follows from Zorn's lemma that both these sets are nonempty. 

\begin{theorem} \label{max} Assume $\Sigma$ is exact and  has property $(DP)$.   

\smallskip \noindent
Then the map $J \to \langle J\rangle$ is a bijection between $\MI(A)$ and $\M\big(C_r^*(\Sigma)\big)$. 

\smallskip \noindent Thus, the family of all simple quotients of $C_r^*(\Sigma)$ is given by $$\Big\{C_r^*\big(\Sigma/J\big)\Big\}_{J \in \MI(A)}$$

\end{theorem}

\begin{proof} 
Let $J\in \MI(A)$. We have to show that $\langle J\rangle \in \M\big(C_r^*(\Sigma)\big)$. We first note that $\langle J\rangle$ is a proper ideal of $C_r^*(\Sigma)$; otherwise, we would have $J = E(\langle J\rangle) = A$, contradicting that $J$ is a proper ideal of $A$. 

Next, let $\K$ be a proper ideal of $C_r^*(\Sigma)$ containing $\langle J\rangle$, and set $K=\overline{E(\K)}$. Since $\Sigma$ has property ($DP$), Proposition \ref{proper} gives that  $K$ is a proper invariant  ideal of $A$. Moreover, we have $J = E(\langle J\rangle) \subset E(\K) \subset K$. By maximality of $J$, we get $J = K$, which gives $$E(\K) = K = J \subset  \langle J\rangle \subset \K\,.$$
Thus, $\K$ is $E$-invariant. Since $\Sigma$ is exact, we get from Lemma \ref{E-inv} that
$\K = \langle K\rangle$. As $J=K$, we conclude that  $\K = \langle J\rangle$. Thus, we have shown that $\langle J\rangle $ is maximal among the proper ideals of $C_r^*(\Sigma)$,
as desired. 

\smallskip This means that the map $J \to \langle J\rangle$ maps $\MI(A)$ into $\M\big(C_r^*(\Sigma)\big)$. This map is clearly  injective (since $E(  \langle J\rangle ) = J$ for every invariant ideal $J$ of $A$).

\smallskip To show that it is surjective,
let  $\J \in \M\big(C_r^*(\Sigma)\big)$ and set $J = \overline{E(\J)}$. We will show that $J \in \MI(A)$ and $\J = \langle J\rangle $.
  
Since $\Sigma$ has property ($DP$) and $\J$ is a proper ideal of $C_r^*(\Sigma)$,  Proposition \ref{proper} gives that $J$ is a proper invariant ideal of $A$. 
Further, since $\Sigma$ is exact, Lemma \ref{basic} gives that $\J \subset \langle J\rangle$.
As $\J$ is maximal, we get $\J=\langle J \rangle$. 

Finally, $J$ is a maximal among the proper invariant ideals of $A$. Indeed, let $K$ be a proper invariant ideal of $A$ containing $J$.  
Then we have \\ $\J = \langle J \rangle \subset \langle K \rangle$. By maximality of $\J$, we get $\langle J \rangle = \langle K \rangle$.  This implies that $J=E(\langle J \rangle) = E(\langle K \rangle) =K$. Hence, we have shown that $J\in \MI(A)$.

\end{proof}

To give examples of systems satisfying property ($DP$), we let $\P$ denote the class of discrete groups consisting of  PH groups \cite{Pro} and of groups satisfying the property ($P_{\rm com}$) introduced in \cite{BCdH}.  
The class $\P$, which is a subclass of the class of discrete C$^*$-simple groups, contains a huge variety of groups,  including for instance many amalgamated free products,  HNN-extensions,  hyperbolic groups, Coxeter groups, and lattices in semisimple Lie groups. For a more precise description, we refer to \cite{dH} (see also \cite{dHP}). 
The following result 
may be seen as a generalization of results in \cite{dHS, BN, Bed, Pro, BCdH}. For the convenience of the reader, we will give a proof in Section \ref{appendix}.

\begin{theorem} \label{P} Let $G\in \P$. Then $\Sigma$ has property $(DP)$.
\end{theorem}

Thus, we get:
\begin{corollary}\label{main} Let $G\in \P$. Then the map $\varphi \to \varphi\circ E$ is a bijection between the set of invariant tracial states of $A$ and the set of tracial states of $C_r^*(\Sigma)$.

Moreover, assume $\Sigma$ is exact. Then the map $J \to \langle J\rangle$ is a bijection between $\MI(A)$ and $\M\big(C_r^*(\Sigma)\big)$. 
Thus, the family of all simple quotients of $C_r^*(\Sigma)$ is given by $$\Big\{C_r^*\big(\Sigma/J\big)\Big\}_{J \in \MI(A)}$$
\end{corollary}
\begin{proof} 
Since $G \in \P$, we know from Theorem \ref{P} that $\Sigma$ has property $(DP)$. 
The result follows therefore from Proposition \ref{tracial} and Theorem \ref{max}. 
 \end{proof}
 
 \begin{corollary} \label{simple2} Assume $G \in \P$.  If $A$ has a unique invariant tracial state, then $C_r^*(\Sigma)$ has a unique tracial state. If $\Sigma$  is minimal, then 
$C_r^*(\Sigma)$ is simple.
\end{corollary}
\begin{proof} This follows  from Proposition \ref{tracial}, Corollary \ref{simple} and  Theorem \ref{P}.
\end{proof}
 
 \begin{corollary}\label{simple3} Let $G \in \P$ and let $\omega \in Z^2(G,\Toro)$. Then $C_r^*(G, \omega) $ is simple with a unique tracial state. 
 
 \end{corollary}
 In fact, proceeding as in the proof of \cite[Corollary 4.10]{Bed} and \cite[Corollary 4]{Bed2}, one sees that Corollary \ref{simple3} holds whenever $G$ is a {\it ultra-$\P$} group, meaning that $G$ has a normal subgroup belonging to $\P$ with trivial centralizer in $G$. Moreover, in the same way,
 one easily deduces that \cite[Corollaries 4.8, 4.9, 4.10, 4.11, 4.12]{Bed} and \cite[Corollaries 5 and 6]{Bed2} still hold if one replaces {\it weak Powers group} by {\it group in the class $\P$}, and {\it ultraweak Powers group} by {\it  ultra-$\P$ group} in the statement of these results.
 
 \medskip It may also be worth mentioning explicitely the following result:
 
 \begin{corollary}\label{main-abelian} Let $G\in \P$ and assume $A$ is abelian, so $A= C(X)$ for some compact Hausdorff space $X$. Then there is a one-to-one correspondence between the set of Borel probability measures on $X$ and  the set of tracial states of $C_r^*(\Sigma)$ given by $\mu \to \int_X E(\cdot) \, d\mu$.

Moreover, assume $\Sigma$ is exact. Then
there is a one-to-one correspondence between  the set  \, $\mathcal{Y}$ of 
 minimal closed invariant subsets of $X$ and $\M\big(C_r^*(\Sigma)\big)$ 
given by $Y \to \Big\langle C_0\big(X\setminus Y\big)\Big\rangle$.
Moreover, 
 the family of all simple quotients of $C_r^*(\Sigma)$ is given by $$\Big\{C_r^*\big(C(Y), G, \alpha_Y, \sigma_Y\big)\Big\}_{Y \in \mathcal{Y}}$$
 where  $(\alpha_Y, \sigma_Y)$ denotes the twisted quotient action of $G$ on $C(Y)$ associated with $(\alpha, \sigma)$. 
 
\end{corollary}
\begin{proof} This follows immediately  from Theorem \ref{main} and Gelfand theory.
\end{proof}

\medskip When $\alpha$ is trivial,  $\sigma$ is just some $2$-cocycle on $G$ with values in $\U(Z(A))$, so $C_r^*(\Sigma)$ is a kind of "twisted" tensor product of $A$ with $C_r^*(G)$. In this case,  we don't have to restrict our attention to maximal ideals of  $C_r^*(\Sigma)$:

\begin{proposition} \label{trivial} Assume $\alpha$ is trivial, $\Sigma$ is exact and $G\in \P$. 
 Then the map $J \to \langle J\rangle$ is a bijection between the set of ideals of $A$ and the set of ideals of $C_r^*(\Sigma)$.
\end{proposition}

\begin{proof} Since $\alpha$ is trivial and $\Sigma$ is exact, it follows immediately from Lemma \ref{E-inv}  that  the map $J \to \langle J\rangle$ is a bijection between the set of ideals of $A$ and the set of $E$-invariant ideals of $B = C_r^*(\Sigma)$. Hence, it suffices to show that any ideal of $B$ is $E$-invariant. 

Let $\cal{J}$ be an ideal of $B$,  \,$y^*=y \in \cal{J}$ and $\varepsilon >0$. Set $x = y - E(y)$. Then $ x^*=x \in B$ and $E(x) = 0$. Since $ G\in \P$, it follows from the proof of Theorem \ref{P} given is Section \ref{appendix} that  there exists a 
$G$-averaging process $\psi$ on  $B$ (as defined in Section \ref{appendix}) such that $\| \psi (x)\| < \varepsilon$. 
Now, since $\alpha$ is trivial, any $G$-averaging process on  $B$  restricts to the identity map on $A$. Thus, we get
$\psi(x) = \psi(y) - \psi(E(y)) = \psi(y) - E(y)$, so 
$$\|  \psi(y) - E(y)\| < \varepsilon\,.$$ 
As any $G$-averaging process on  $B$ preserves ideals, we have $\psi(y) \in \cal{J}$. Hence, we get $E(y) \in \overline{\cal{J}}= \cal{J}$. It clearly  follows that $\cal{J}$ is $E$-invariant, as desired.

\end{proof}

\section{Examples}\label{exa}

This section is devoted to the discussion of some concrete examples.

\medskip

\subsection{Ê}Ê 
As a warm-up, we consider the simple, but instructive case of an action of  a group $G$ on a nonempty finite (discrete) set $X$ 
with $n$ elements. 
Let $\alpha$ denote the associated action of $G$ on $A= C(X) \simeq \Complessi^n$ 
and $\sigma \in Z^2(G,\Toro)$.  

We  may then pick $x_1, \ldots, x_m \in X $ such that $X$ is the disjoint union  
of the orbits $O_j = \big\{ g\cdot x_j\mid g\in G\big\}$ for $ j = 1, \ldots , m$. Clearly, the $O_j$'s are the minimal (closed) invariant subsets of $X$. Hence, 
if $G$ is an exact group in the class $\P$, we get from Corollary  \ref{main-abelian} that the simple quotients of $B =C_r^*\big(C(X), G, \alpha, \sigma \big)$ are given by 
$$ B_j =C_r^*\big(C(O_j), G, \alpha_j, \sigma \big) \, , \,\, j=1, \ldots, m\,,$$
where $\alpha_j $ is the action on $C(O_j)$ obtained by restricting $\alpha$ for each $j$. 

The assumption above that $G$ is exact is in fact not necessary. Indeed, one easily sees that  $B$ is the direct sum of the $B_j$'s. So if $G$ belongs to $\P$, then Corollary \ref{simple2} gives that all the $B_j$'s are simple, and the same assertion as above follows readily.

Finally, assume that $\sigma =1$. Then this characterization of the simple quotients of $B$ still holds whenever $G$ is a C$^*$-simple group. Indeed, letting $G_{x_j}$ denotes the isotropy group of $x_j$ in $G$ and identifying $O_j$ with $G/G_{x_j}$, one gets from \cite[Example 6.6]{Ec} (see also \cite{QS, KW}) that
each $B_j$ is Morita equivalent to $C_r^*(G_{x_j})$. Now, if $G$ is C$^*$-simple, then each $C_r^*(G_{x_j})$ is simple (i.e.\ $G_{x_j}$ is C$^*$-simple) because $G_{x_j}$ has finite index in $G$ (cf.\ \cite{dH} and \cite{Po}), so the $B_j$'s are the simple quotients of $B$. 

\subsection{}
Consider the canonical action ${\rm lt}$ of a group $G$ by left translation on $\ell^\infty(G)$, in other words, the action associated with the natural left action  of $G$ on its Stone-$\check{\text{C}}$ech compactification $\beta G$ \cite{Ell, HiS}, and let $\sigma \in Z^2(G, \Toro)$. 
 
 It is known  that $\beta G$ has $2^{2^{|G|}}$ minimal closed invariant subsets (see for instance \cite[Theorem 1.4]{HLS}Ê and \cite[Lemma 19.6]{HiS}). Moreover, all these subsets are $G$-equivariantly homeomorphic to each other (this follows from \cite[Theorem 19.8]{HiS}). Hence, letting $X_G$ denote one of these minimal closed invariant subsets, we get from Corollary \ref{main-abelian} that if $G$ is exact and belongs to $\P$, then the simple quotients of the ''twisted'' Roe algebra $C_r^*\big(\ell^\infty(G), G, {\rm lt}, \sigma\big)$ are all isomorphic to $C_r^*\big(C(X_G), G, {\rm lt}, \sigma\big)$. 

In general, if $G$ is exact and we assume that $\sigma=1$, one may in fact deduce that there is a one-to-correspondence between the set of all invariant closed subsets of $\beta G$ and the ideals of the Roe algebra $C_r^*(\ell^\infty(G), G, {\rm lt})$; indeed, since the action of $G$ on $\beta G$ is known to be free \cite[Proposition 8.14]{Ell}, this follows from   \cite[Theorem 1.20]{S}. 
 
\subsection{}  Let $\Gamma = \Relativi^3 \rtimes SL(3, \Relativi)$ be the semidirect product of $\Relativi^3$ by the canonical action of $SL(3, \Relativi)$. 
Since $ \Relativi^3$ is  a normal nontrivial amenable subgroup of $\Gamma$, it is well known that $\Gamma$ is not C$^*$-simple. In aim to describe the maximal ideals of  $C_r^*(\Gamma)$, 
we  decompose  $$C_r^*(\Gamma) \simeq C_r^*\big(C_r^*(\Relativi^3), SL(3, \Relativi), \alpha\big) \simeq C_r^*\big(C(\Toro^3), SL(3, \Relativi), \tilde\alpha\big)$$ where $\alpha$ (resp.\ $\tilde\alpha$) denotes the associated action of $SL(3, \Relativi)$  on $C_r^*(\Relativi^3)$ (resp.\ $C(\Toro^3)$). Now, $SL(3,\Relativi)$ is exact \cite{BrOz} and belongs to $\P$ (since it has property ($P_{\rm com}$) \cite{BCdH}).
Hence, appealing to Corollary \ref{main-abelian}, the maximal ideals of  $C_r^*(\Gamma)$ are in a one-to-one correspondence with the minimal closed invariant subsets of $\Toro^3$.  The orbits of the action of  $SL(3, \Relativi)$ on $\Toro^3$ are either finite or dense (see for instance \cite{Muc, GS}), hence the minimal closed invariant subsets of $\Toro^3$ are the orbits of rational points in $\Toro^3=\Reali^3/\Relativi^3$.

Let $x \in \mathbb{Q}^3/\Relativi^3\subset \Toro^3$ and let $G_x$ denote the isotropy group of $x$ in $G= SL(3, \Relativi)$. Then identifying the (finite) orbit $O_x$ of $x$ in $\Toro^3$ with $G/G_x$, we get that the simple quotient $B_x$ of $C_r^*(\Gamma)$ corresponding to $O_x$ is given by 
the reduced crossed product $$B_x = C_r^*( C(O_x), G, \alpha^x) \simeq C_r^*( C(G/G_x), G, \beta^x)$$ where $\alpha^x$ is implemented by the action of $G$ on $O_x$  and $\beta^x$ is implemented by the canonical left action of $G$ on $G/G_x$. We note that $B_x$ has a unique tracial state since $G $ belongs to $\P$ and there is obviously only one invariant state on $C(O_x)$. Moreover, it follows from \cite[Example 6.6]{Ec} (see also \cite{QS, KW}) that $B_x$ is Morita equivalent to $C_r^*(G_x)$.  This implies that $G_x$ is C$^*$-simple,  a fact that may also be deduced from \cite{dH} (see also \cite{Po}) since $G_x$ has finite index in $G$. 

\subsection{} Let $\Gamma$ be an exact discrete group such that $G=\Gamma/Z$ belongs to the class $\P$, where $Z=Z(\Gamma)$ denotes the center of $\Gamma$. We can then easily deduce that the ideals of $C_r^*(\Gamma)$ are in a one-to-one correspondence with the open (resp.\ closed) subsets of the dual group $\widehat{Z}$. Indeed, using \cite[Theorem 2.1]{Bed}, we can decompose 
$$C_r^*(\Gamma)\simeq C_r^*\big(C_r^*(Z), G, {\rm id}, \omega\big) \simeq C_r^*\big(C(\widehat{Z}), G, {\rm id}, \widehat{\omega}\big)$$  
where $\omega: G\times G \to \U\big(C_r^*(Z)\big)$ is given by $$\omega(g,h)= \lambda_Z\big(n(g)n(h)n(gh)^{-1}\big) \quad (g,h \in G)$$ for some section $n:G\to \Gamma$ of the canonical homomorphism $q: \Gamma \to G$ such that $n(e_G) = e_\Gamma$, while the second isomorphism is implemented by Fourier transform. So the assertion follows from Gelfand theory and Proposition \ref{trivial}.

Some specific examples are as follows: 

\begin{itemize}
\item{} Consider $\Gamma = SL(2n, \Relativi)$ for some $n\in \Naturali$. Then $Z=Z(\Gamma)\simeq \Relativi_2$ and $G = \Gamma/Z= PSL(2n, \Relativi)$ is exact (cf.\ \cite[Section 5.4]{BrOz}) and belongs to $\P$ (cf.\ \cite{BCdH}). Hence, we get that  $C_r^*\big( SL(2n, \Relativi)\big)$ has two nontrivial ideals. 
\item{} Consider the pure braid group $\Gamma =P_n$ on $n$ strands for some $n \geq 3$. Then $Z_n:=Z(P_n) \simeq \Relativi$ and $G = P_n/Z_n$ is a weak Powers group (cf.\  \cite{GH} and \cite{BN}).  Moreover $P_n$ is exact; this follows by induction on $n$, using the exact sequence
$$1 \rightarrow \mathbb{F}_{n-1} \rightarrow P_n/Z_n \rightarrow P_{n-1}/Z_{n-1} \rightarrow 1$$
(cf.\ \cite[Proposition  6]{GH}, where $P_2=Z_2 = 2\Relativi$) and the fact that extension of exact groups are exact (cf.\ \cite[Proposition 5.11]{BrOz}.
Hence, we obtain that the ideals of $C_r^*(P_n)$ are in a one-to-one correspondence with the open (resp.\ closed) subsets of 
$\Toro$.

\item{}ÊConsider the braid group $\Gamma = B_3$ (i.e.\ the trefoil knot group). Then, $Z=Z(\Gamma) \simeq \Relativi$,  and $G= \Gamma/Z \simeq \Relativi_2 * \Relativi_3 \simeq PSL(2, \Relativi)$ belongs to $\P$. As, by definition of $P_3$, we have an exact sequence $1\to P_3\to B_3\to S_3 \to 1$, where $S_3$ denotes the symmetric group on three symbols, it follows that $B_3$ is exact. (This also follows from the fact that braid groups are known to be linear groups).  Hence, we get that the ideals of $C_r^*(B_3)$ are in a one-to-one correspondence with the open (resp.\ closed) subsets of $\Toro$.

\end{itemize} 

If one considers the braid group $B_n$ on $n$ strands for $n \geq 4$, then we believe that one should arrive at the same result as the one for $B_3$, but we don't know for the moment whether $B_n/Z_n$ belongs to the class $\P$; $B_n/Z_n$ is known to be a ultraweak  Powers group \cite[p.\ 536]{Bed}), and Promislow has a result indicating that ultraweak Powers groups might be PH groups (see \cite[Theorem 8.1]{Pro}), but this is open in general. 

\newpage 
\section{Proof of Theorem \ref{P}} \label{appendix}

We start by representing  $B=C_r^*(\Sigma)$ faithfully on a Hilbert space. 
Without loss of generality, we may assume that $A$ acts faithfully on a Hilbert space $\H$, and let $(\pi, \lambda)$ be any regular covariant representation of $\Sigma$ on the Hilbert space $\ell^2(G,\H)$;  as in \cite{Bed}, we will work with the one  defined by
\[ \big(\pi (a)\xi\big)(h) = \alpha_{h^{-1}}(a)\,\xi(h), \]
\[ \big(\lambda(g)\xi\,\big)(h) = \sigma(h^{-1},g)\,\xi(g^{-1}h), \]
 for $a\in{A},\,\xi \in \ell^2(G,{\H}),\, h, g \in G$. 
 
 \medskip We may then identify $B$ with $C^*\big(\pi(A), \lambda(G)\big)$. The canonical conditional expectation from $B$ onto $\pi(A)$ will still be denoted by $E$. When $x \in B$, we set ${\rm supp}(x) = \{ g \in G \mid \widehat{x} (g) \neq 0\}$, where $\widehat{x}(g) = E(x\, \lambda(g)^*)$.
 We will let $B_0$ denote the dense $^*$-subalgebra of $B$ generated by $\pi(A)$ and $\lambda(G)$. So if $x \in B_0$, we have
$$x \, = \sum_{g \in {\rm supp}(x) } \, \widehat{x}(g) \, \lambda(g) \quad \text{(finite sum)} \,.$$
 If $D \subset G$, we let $P_D$ denote the orthogonal projection from $\ell^2(G,\H)$ to $\ell^2(D,\H)$ (identified as a closed  subspace of $\ell^2(G,\H)$).
 
 \medskip 
 \smallskip Moreover, if $F \in \ell^\infty\big(G, \B(\H)\big)$, that is, $\, F : G \to \B(\H)$ is a map satisfying  $\|F\|_\infty:= \sup_{h\in G} \|F(h)\| < \infty$,  we let $M_F\in \B\big(\ell^2(G,\H)\big)$ be defined by
 $$(M_F\,\xi) (h) = F(h)\, \xi(h)\, , \quad \xi \in \ell^2(G,\H), \, h \in G\,,$$
 noting that $\|M_F\| = \|F\|_\infty  < \infty$. 
 
 \medskip We remark that if $a \in A$ and we let $\pi_a: G \to \B(\H)$ be defined by $\pi_a(h) =  \alpha_{h^{-1}}(a)$ for each  $\, h \in H$, then $\pi_a \in  \ell^\infty\big(G, \B(\H)\big)$ and $M_{\pi_a} = \pi(a)$. 
 
 \medskip Straightforward computations  give that for $F \in \ell^\infty\big(G, \B(\H)\big)$, $D \subset G$ and $g\in G$, we have
  \begin{equation}\label{eq-projM}
M_F\,P_D = P_{D}\, M_F\,, \quad
 \lambda(g)\,P_D = P_{gD}\, \lambda(g)\,.
 \end{equation}
 In passing, we remark that we also have $
 \lambda(g)\,M_F\, \lambda(g)^* = M_{F_g}\,,$ where 
 \begin{equation*}
 F_g(h) = \sigma(h^{-1}, g)\, F(g^{-1}h) \, \sigma(h^{-1}, g)^*\,.
 \end{equation*}
 As a sample, we check that the second equation in (\ref{eq-projM}) holds. 
 
 \medskip Let $\xi \in \ell^2(G,\H)$ and $h \in G$. Then we have
 $$ [(\lambda(g)\,P_D) \xi](h) =  \sigma(h^{-1},g)(P_D \xi)(g^{-1}h)= \left\{ \begin{array}{cl}
 \sigma(h^{-1},g)\, \xi(g^{-1}h) & \mbox{\; if   $g^{-1}h \in D$,}\\
&\\
\; 0 & \mbox{\, if $g^{-1}h \not\in D$}
\end{array} \right.$$
 
 $$= \left\{ \begin{array}{cl}
  \sigma(h^{-1},g)\, \xi(g^{-1}h) & \mbox{if $h \in gD$,}\\
&\\
\; 0 & \mbox{\, if $h \not\in gD$}
\end{array} \right.=  \left\{ \begin{array}{cl}
 (\lambda(g) \xi)(h) & \mbox{if   $h \in gD$,}\\
&\\
\; 0 & \mbox{if $h \not\in gD$}
\end{array} \right. 
$$ 
$\qquad = [(P_{gD}\,\lambda(g)) \xi](h)
$, as desired.

\bigskip Let $H$ be a subgroup of $G$. By a {\em simple $H$-averaging process} on $B$, we will mean
 a linear map $\phi:{B}\rightarrow{B}$ such that there exist
$n\in \Naturali$ and $h_1,\ldots,h_n\in H$ satisfying
\[ \phi(x) = \frac{1}{n} \sum_{i=1}^n \lambda(h_i)\,x\,\lambda(h_i)^* \quad
\mbox{for all $\, x \in B$}. \]
Moreover, a {\em $H$-averaging process on $B$\/} is a linear map 
$\psi:{ B}\rightarrow{ B}$ such that there exist
$m \in {\Naturali}$ and $\phi_1,\ldots,\phi_m$ simple $H$-averaging processes on
$ B$ with $\psi=\phi_m\circ\phi_{m-1}\circ\ldots\circ\phi_1$. 

\medskip Let $\U_G$ denote the subgroup of $\U(B)$ generated by the $\lambda(g)$'s and let $\psi$ be a $G$-averaging process on $B$. Clearly, for all $x \in B$, we then have $$\psi(x) \in \text{co}\big\{v\, x \, v^*\, |Ê\, v \in \U_G\big\}\,.$$
Hence, to show that $\Sigma$ has (the strong) property (DP), it suffices to show that for every $x^*=x \in B$ satisfying $E(x) = 0$ and every $\varepsilon > 0$, there exists a $G$-averaging process $\psi$ on $B$ such that $\|\psi(x) \| < \varepsilon$. 

\smallskip In fact, it suffices to show the last claim for every $x^*=x \in B_0$ satisfying $E(x) = 0$ and every $\varepsilon > 0$. Indeed, assume that this holds and consider some $b^*=b \in B$ satisfying $E(b) =0$ and $\varepsilon >0$. Then pick  $y^*=y \in B_0$ such that $\|b-y\|\leq \varepsilon/3$, and set $x = y -E(y)$. Then $x^*=x \in B_0$ and $E(x)=0$, so we can find a $G$-averaging process on $B$ such that 
$\|\psi(x) \| < \varepsilon/3\,.$ Since $\|E(y)\| = \|E(y-b)\| \leq \|y-b\| < \varepsilon/3\,$, we get
$$\|\psi(b)\| \leq \|Ê\psi(b-y)\|Ê+ \|Ê\psi(y-E(y))\| + \|\psi(E(y))\|$$ $$ \leq \|b-y\| + \|\psi(x)\| + \|E(y)\| < \varepsilon\,,$$
as desired.

\subsection{Ê} In this subsection we will prove that Theorem \ref{P} holds when $G$ is a PH group, as defined in \cite{Pro}. 
We first recall the definition of a PH group. 

\medskip If $g \in G$ and $A \subset G$, then set $$<\!g\!>_A= \{aga^{-1}Ê\mid a \in A\}\,.$$ Now, if $T \subset G$ and 
$\emptyset \neq M \subset G\setminus \{e\}$, then $T$ is said to be $M$-{\it large} (in $G$) if 
$$m (G\setminus T) \subset T \quad \text{for all}\, \, m \in M\,.$$ 
Further, let $\emptyset \neq F \subset G\setminus \{e\}$ and $H \subset G$. Then $H$ is said to be a {\it Powers set for } $F$ if, for any $N\in \Naturali$, there exist $h_1, \ldots, h_N \in H$ and pairwise disjoint subsets $T_1, \ldots, T_N$ of $G$ such that $T_j$ is $h_jFh_j^{-1}$-large for $j=1, \ldots, N$. Moreover, if $g \in G\setminus\{e\}$, then $H$ is said to be a {\it c-Powers set for} $g$ if $H$ is a Powers set for $<\!g\!>_M$ for all finite, non-empty subsets $M$ of $H$. 

\medskip If $G$ is a  weak Powers group (see \cite{BN, Bed, dH}), then  $G$ is a c-Powers set for any $g \in G\setminus \{e\}$.
More generally, $G$ is said to be a {\it PH group} if, given any finite non-empty subset $F$ of  $G\setminus \{e\}$, one can write
$F=\{ f_1, f_2, \ldots, f_n\}$ and find a chain of  subgroups $G_1 \subset G_2 \subset \cdots \subset G_n \subset G$ 
 such that $G_j$ is a c-Powers set for $f_j, \, j=1, \ldots, n$.  
 
 Note that in his definition of a PH group, Promislow just requires that one can find a chain of  subsets $e\in G_1 \subset G_2 \subset \cdots \subset G_n$ of $G$  such that $G_j$ is a c-Powers set for $f_j, \, j=1, \ldots, n$. Requiring these subsets to be subgroups of $G$ (or at least subsemigroups) seems necessary to us for the proof of his main result, \cite[Theorem 5.3]{Pro}, to go through.  We will use the subsemigroup property  in the proof  of Lemma \ref{H-ave}. 
 
 \medskip The class of PH groups has the interesting property that it closed under extensions \cite[Theorem 4.6]{Pro}\footnote{One easily checks that all the results in \cite{Pro} are still true under our slightly more restrictive definition.}. For example, an extension of a weak Powers group by a weak Powers group is a PH group (but not necessarily a weak Powers group). 

\medskip We will need a lemma of de la Harpe and Skandalis (\cite[Lemma 1]{dHS}; see also \cite[Lemma 4.3]{Bed}) in a slightly generalized form. For completeness, we include the proof, which is close to the one given in \cite{dHS}.

\begin{lemma} \label{HS} Let $\H$ be a Hilbert space and $x^*=x\in \B(\H)$. Assume that there exist orthogonal projections $p_1, p_2, p_3$  and unitary operators  $u_1, u_2, u_3$ on $\H$ such that $$p_1\,x\,p_1 = p_2\,x\,p_2  = p_3\,x\,p_3 = 0\,  $$ and $u_1(1-p_1)u_1^*, \,u_2(1-p_2) u_2^*, \, u_3(1-p_3) u_3^*$ are pairwise orthogonal. Then we have 
$$\Big\|\,\frac{1}{3}\sum_{j=1}^3 u_j\, x \, u_j^*\Big\| \, \leq \, \Big(\frac{5}{6} + \frac{\sqrt{2}}{9}\Big) \, \|x\| < 0.991\, \|x\|$$

\end{lemma} 
\begin{proof}  Without loss of generality, we may clearly assume that $\|x\| = 1$. 

\medskip \noindent Set $y= \frac{1}{3}\sum_{j=1}^3 u_j\, x \, u_j^*$ and $ q_j = u_j\,(1-p_j)\, u_j^*, \, j=1, 2, 3$. 

\medskip Let $\xi \in \H, \, \|\xi\|= 1$. 
Since the $q_j$'s are pairwise orthogonal, there exists an index $j$ such that $\|Êq_j\, \xi\|^2 \leq 1/3$. We may assume that $j=1$, and set $\xi_1 = u_1^*\, \xi$.

\medskip As $\| (1-p_1)\, \xi_1\|^2 = \|q_1 \,\xi\|^2 \leq 1/3$, one has $$\|p_1 \xi_1\|^2 \geq 2/3\quad \text{and}\, \quad   
\| p_1\, x\, (1-p_1)\, \xi_1\|^2 \leq 1/3\, .$$
Now, since $p_1\, x \, p_1 = 0$ by assumption, we get $$\|x \, \xi_1 -\xi_1\| \geq \|p_1\, \xi_1 - p_1\, x \,\xi_1\| = \|p_1\, \xi_1 - p_1\, x \,(1-p_1)\xi_1 -  p_1\, x \,p_1\xi_1\|$$
$$\geq \big| \, \|p_1\, \xi_1\| - \|p_1\, x \,(1-p_1)\xi_1\| \, \big| \geq \frac{\sqrt{2} -1}{\sqrt{3}}$$
As $\|x \, \xi_1 -\xi_1\|^2 \,\leq\, 2\, \big(1 - \langle x \, \xi_1, \xi_1\rangle\big)$, it follows that
$$  \big\langle x \, \xi_1, \xi_1\big\rangle\,  \leq \, 1- \frac{1}{2} \|x \, \xi_1 -\xi_1\|^2 \leq 1 - \frac{1}{2} \Big(\frac{\sqrt{2} -1}{\sqrt{3}}\Big)^2 = \frac{3 + 2 \sqrt{2}}{6}$$
So, using the Cauchy-Schwarz inequality, we get $$ \big\langle y \, \xi, \xi\big\rangle\, \leq \, \frac{1}{3} \, \big\langle x \, \xi_1, \xi_1\big\rangle +\frac{2}{3}\, \leq \frac{1}{3}Ê\, \Big( \frac{3 + 2 \sqrt{2}}{6} + 2\Big) =   \frac{5}{6} + \frac{\sqrt{2}}{9}< 0.991\,.$$
The same argument with $-x$ gives 
$$\Big|\, \big\langle y \, \xi, \xi\big\rangle\, \Big| \, \leq \, \frac{5}{6} + \frac{\sqrt{2}}{9}< 0.991\,.$$
Since $y$ is self-adjoint, taking the supremum over all $\xi\in \H$ such that $\|\xi\|= 1$, we obtain 
$$ \|y\| \leq \, \frac{5}{6} + \frac{\sqrt{2}}{9}< 0.991\, ,$$
as desired. 
\end{proof}

\begin{lemma} \label{crucial} Let $x^*=x \in B_0$ satisfy $E(x) = 0$ and set $S = {\rm supp}(x)$. 
Assume that  $S \subset F \cup F^{-1}$ for some finite subset $F$ of $G\setminusÊ\{e\}$ and  that there exists a subgroup $H$ of  $G$ which is a Powers set for $F$. 

Then there exists an $H$-simple averaging process $\phi$ on $B$ such that $$\|\phi(x)\| < 0.991 \, \|x\|\,.$$

\end{lemma}

\begin{proof} One easily sees that $H$ is also a Powers set for $S= F \cup F^{-1}$ (cf.\ \cite[Lemma 2.1]{Pro}). Pick $h_1, h_2, h_3 \in H$ and 
 pairwise disjoint subsets $T_1, T_2, T_3$ of $G$ such that $T_j$ is $h_jSh_j^{-1}$-large for $j=1, 2, 3$.
 
 \medskip For each $j$, set $E_j = h_j ^{-1} T_j$, \, $D_j = G\setminus E_j$ and let $p_j$ be the orthogonal projection from  $\ell^2(G, \H)$ onto $\ell^2(D_j, \H)$. Then we have $p_j \, x\, p_j = 0$ for each $j$. Indeed, one easily checks that 
$T_j$ is $h_jSh_j^{-1}$-large means that 
$$s \, D_j\cap D_j = \emptyset \quad \text{for every } \, s \in S\,.$$ Hence, using the identities in  (\ref{eq-projM}), the above assertion readily follows. 
 
 \medskip Moreover, set $q_j = \lambda(h_j) \, (1-p_j) \lambda(h_j)^*$. Then  $q_j $ is the orthogonal projection from $\ell^2(G, \H) $ onto $\ell^2(h_jE_j, \H) = \ell^2(T_j, \H)$. Since the $T_j$'s are pairwise disjoint, the $q_j$'s are pairwise orthogonal. Thus, we can apply Lemma \ref{HS} and conclude that 
$$\Big\|\,\frac{1}{3}\sum_{j=1}^3 \lambda(h_j)\, x \, \lambda(h_j)^*\Big\| \,  <\,  0.991\, \|x\|$$
which shows the assertion.

\end{proof}

\begin{lemma} \label{H-ave} 
Let $\delta >0$, $g \in G\setminusÊ\{e\}$ and assume that there exists a subgroup $H$ of $G$ which is a c-Powers set for $g$.
Let $x^*=x \in B_0$ satisfy $${\rm supp}(x) \, = \, \,<\!g\!>_M \cup <\!g^{-1}\!>_M$$ for some finite
non-empty subset $M$ of $H$.

Then there exists an $H$-averaging process $\psi$ on $B$ such that $\|\psi(x)\| < \delta$. 

\end{lemma} 
\begin{proof}
By assumption, $H$ is a Powers set for $<\!g\!>_M $. Applying Lemma \ref{crucial} (with $F=<\!g\!>_M$), we get that there exists an  $H$-simple averaging process $\phi_1$ on $B$ such that $\|\phi_1(x)\| < d \, \|x\|\,,$  where $d=0.991$. 
Now, one easily checks (cf.\ \cite[Lemma 4.4]{Bed}) that $${\rm supp}\big(\phi_1(x)\big) \, \subset \, \,Ê <\!g\!>_{M_1}\, \cup \, <\!g^{-1}\!>_{M_1}\,,$$ where $M_1$ is a finite non-empty subset of $H$ (since $H$ is closed under multiplication, being a subgroup). Moreover, $\phi_1(x)$ is a selfadjoint element of $B_0$ satisfying $E(\phi_1(x))=0$. Hence we can apply Lemma \ref{crucial} (with $F=<\!g\!>_{M_1}$) and get that there exists an  $H$-simple averaging process $\phi_2$ on $B$ such that 
$$\|\phi_2(\phi_1(x))\| \, < \,d \, \|\phi_1(x)\| \,\leq \, d^{\,2} \, \|x\|\,.$$ 
Iterating this process, we get that for each $k\in \Naturali$, there exist $H$-simple averaging processes $\phi_1, \ldots, \phi_k$ on $B$
such that  $$\|(\phi_k \circ \cdots \phi_1)(x))\| \, <  \, d^{\,k} \, \|x\|\,.$$ 
Choosing $k$ such that $d^{\,k} < \delta$ gives the result.

\end{proof}

\begin{theorem} \label{PH} Assume $G$ is a PH group. Then $\Sigma$ has property (DP). 

\end{theorem}

\begin{proof}
Let $x^*=x \in B_0$ satisfy $E(x) = 0$, and let $\varepsilon > 0$. Write $S= {\rm supp}(x)$ as a disjoint union $S= E \cup F \cup F^{-1}$ where $E=\{ s \in S\mid s^2 = e\}$. 

Consider $E\cup F \subset G\setminus\{e\}$. Since $G$ is a PH group, we can write
$E\cup F=\{ s_1, s_2, \ldots, s_n\}$ and find a chain of  subgroups $G_1 \subset G_2 \subset \cdots \subset G_n \subset G$ 
 such that $G_j$ is a c-Powers set for $s_j, \, j=1, \ldots, n$. Thus, each $G_j$ is a Powers set for $<\!s_j\!>_M$, for all finite subsets $M$ of $G_j$. 
 
 \medskip Write $x = \sum_{j=1}^n \, x_j$\,, where $x_j^*=x_j \in B_0$ and ${\rm supp}(x_j)= \{s_j\} \cup \{s_j^{-1}\}$ for each $j$. (Note that if $s_j \in E$, we have  $s_j^{-1}=s_j$, so ${\rm supp}(x_j)= \{s_j\} $).

\medskip Since ${\rm supp}(x_1) =  \,<\!s_1\!>_M \cup <\!s_1^{-1}\!\!>_M$, with $M=\{e\} \subset G_1$, and $G_1$ is a c-Powers set for $s_1$, 
Lemma \ref{H-ave} applies and gives that there exists a $G_1$-averaging process $\psi_1$ on $B$ such that $\|\psi_1(x_1)\| < \varepsilon/n$. 

Now, consider $\tilde{x}_2 = \psi_1(x_2)$. Then ${\rm supp}(\tilde{x}_2) =  \,<\!s_2\!>_M \cup <\!s_2^{-1}\!\!>_M$ for some finite subset $M$ of $G_1$. Since $G_1$ is contained in $G_2$, and $G_2$ is  a c-Powers set for $s_2$, 
Lemma \ref{H-ave} applies again and gives that there exists a $G_2$-averaging process $\psi_2$ on $B$ such that $\|\psi_2(\tilde{x}_2)\| < \varepsilon/n$, that is, $\|(\psi_2\circ \psi_1)({x}_2)\| < \varepsilon/n$.

Proceeding inductively, let $1 \leq k \leq n-1$ and assume that for each $j=1, \ldots, k$, we have constructed a $G_j$-averaging process $\psi_j$ on $B$, such that $\|(\psi_j \circ \cdots  \circ \psi_1)(x_j)\| < \varepsilon/n$ for $j=1, \ldots, k$. Then consider 
$\tilde{x}_{k+1} = (\psi_k \circ \cdots  \circ \psi_1)(x_{k+1})$. Then ${\rm supp}(\tilde{x}_{k+1}) =  \,<\!s_{k+1}\!>_M \cup <\!s_{k+1}^{-1}\!\!>_M$ for some finite subset $M$ of $G_k$. Since $G_k$ is contained in $G_{k+1}$, and $G_{k+1}$ is  a c-Powers set for $s_{k+1}$, 
Lemma \ref{H-ave} applies and gives that there exists a $G_{k+1}$-averaging process $\psi_{k+1}$ on $B$ such that $\|\psi_{k+1}(\tilde{x}_{k+1})\| < \varepsilon/n$, that is, $\|(\psi_{k+1} \circ \cdots \circ \psi_1)(x_{k+1})\| < \varepsilon/n$.

Repeating this until $k=n-1$, we obtain, for each $1\leq j \leq n$, a $G_j$-averaging process $\psi_j$ on $B$ such that $\|(\psi_j \circ \cdots  \circ \psi_1)(x_j)\| < \varepsilon/n$. 
Set $\psi= \psi_n \circ \cdots  \circ \psi_1$. Then $\psi$ is a $G$-averaging process on $B$ and, for each $1\leq j \leq n$, we have 
$$\|\psi(x_j)\| = \|(\psi_n\circ \cdots \circ \psi_{j+1}\circ \psi_j\circ \cdots \circ \psi_{1})(x_j)\| \leq \| (\psi_j\circ \cdots \circ \psi_{1})(x_j)\| < \varepsilon/n\,,$$ so we get
$$\|\psi(x) \|  \leq \sum_{j=1}^{n} \|\psi(x_j)\| < \varepsilon\,.$$
This shows that $\Sigma$ satisfies (the strong) property DP.

\end{proof}

\subsection{} We now turn to the proof 
that $\Sigma$ has property (DP) when $G$ satisfies property ($P_{\rm com}$). We will adapt the arguments given in \cite{BCdH} to cover the twisted case.  We recall from \cite{BCdH} that $G$ {\it is said to have property $(P_{\rm com})$} when
the following holds:
\medskip

\noindent
Given any non-empty finite subset $F \subset G \setminus\{e\}$,  there exist $n\in \Naturali$, $g_0 \in G$ and subsets $U, \, D_1,\, \ldots, \, D_n$ of $G$ such that 
\begin{description}
\item (i) \@
$G\setminus U \, \subset \, D_1\cup \cdots \cup D_n\, $,
\item (ii) \@
$g\,U\cap U = \emptyset\;$ for all $g \in F\, $,
\item (iii) \@
$g_0^{-j}D_k \cap D_k = \emptyset$ for all $j \in \Naturali$ and $k=1,\ldots,n\,$.
\end{description}

 \begin{lemma} \label{fund} $(cf.\ \cite{BCdH})$. Let $ g\in G \setminus\{e\}$ and assume there exist $n\in \Naturali$ and subsets $U, \, D_1,\, \ldots, \, D_n$ of $G$ such that 
 $$G\setminus U \, \subset \, D_1\cup \cdots \cup D_n\quad \text{and} \quad g\,U\cap U = \emptyset\,. $$
Let $F \in \ell^\infty\big(G, \B(\H)\big)$ and $\xi, \eta \in \ell^2(G,\H)$. Then we have
\begin{equation}\label{estim}
|\, \big\langle\,M_F \lambda(g)  \xi\,, \, \eta\,\big \rangle\, | \, \leq \sum_{j=1}^n\Big( \|M_F\lambda(g)\xi\| \, \|P_{D_j}\eta\| + \|P_{D_j}\xi\|  \,\|  M_F^{\,*}\eta\|\Big)
\end{equation}

 \end{lemma}
 \begin{proof} We set $V=G\setminus U$, and note that $P_{\,U} \,P_{g\,U} = P_{\,U\,\cap\, g\,U} = 0$. Thus, making use of (\ref{eq-projM}), we get
 $$\big\langle\,M_F \lambda(g)  \xi\,, \, \eta\, \big\rangle =\, \big\langle\,M_F \lambda(g)P_{\,U}\,  \xi\,, \, \eta\, \big\rangle + \big\langle\,M_F \lambda(g) P_{\,V}\, \xi\,, \, \eta\, \big\rangle$$
 $$ = \,\big\langle\,P_{g\,U}M_F \lambda(g)\,  \xi\,, \, (P_{\,U} + P_{\,V})\,\eta\, \big\rangle + \big\langle\,\lambda(g)P_{V} \,  \xi\,, \, M_F^{\,*}\, \eta\, \big\rangle $$
 $$ = \,\big\langle\,P_{g\,U}M_F \lambda(g) \, \xi\,, \,  P_{\,V}\,\eta\, \big\rangle + \big\langle\,\lambda(g)P_{V} \,  \xi\,, \, M_F^{\,*} \eta\, \big\rangle \,.$$
 Thus, the triangle inequality and the Cauchy-Schwarz inequality give
 $$|\, \big\langle\,M_F \lambda(g) \, \xi\,, \, \eta\,\big \rangle\, | \, \leq 
  \,|\,\big\langle\,P_{g\,U}M_F \lambda(g) \, \xi\,, \,  P_{\,V}\,\eta\, \big\rangle\,|\, + \, |\,\big\langle\,\lambda(g)P_{V} \,  \xi\,, \, M_F^{\,*} \eta\, \big\rangle\,|$$
  $$\leq \|M_F \lambda(g) \, \xi\| \, \|P_{\,V}\,\eta\| + \|P_{V} \,  \xi\| \, \|M_F^{\,*} \eta\|$$
 $$ \leq \sum_{j=1}^n\Big( \|M_F\lambda(g)\,\xi\| \, \|P_{D_j}\eta\| + \|P_{D_j}\xi\|  \,\|  M_F^{\,*}\eta\|\Big)\, $$
 since $\|P_{\,V} \zeta\| \leq \sum_{j=1}^n \|P_{D_j}\zeta\| $  for any $\zeta \in \ell^2(G,\H)$, as is easily checked, using that $V\, \subset \, D_1\cup \cdots \cup D_n$.
 
  \end{proof}
  \begin{lemma}\label{D-ineq} Let $D\, \subset G$, $\zeta \in \ell^2(G,\H)$ and assume there exist $N\in \Naturali$ and $g_1, \ldots, g_N \,\in\, G$ such that 
  $g_1 D, \, \ldots, \, g_N D$ are pairwise disjoint. Then we have
  $$ \sum_{j=1}^N \, \|P_{\,g_jD}\, \zeta\| \, \leq \, \sqrt{N} \, \|\zeta\|\,.$$
  \end{lemma}
\begin{proof}
The Cauchy-Schwarz inequality and the assumption give
$$ \sum_{j=1}^N \, \|P_{\,g_jD}\, \zeta\| \, \leq  \, \sqrt{N} \, \Big[\sum_{j=1}^N \, \|P_{\,g_jD}\, \zeta\|^2\Big]^{1/2}\,= \,     \sqrt{N} \,\Big[\,\sum_{h \,\in \, g_1D\, \cup \cdots \cup \,g_N D} \,  \|\zeta(h)\|^2\,\Big]^{1/2}  
$$$$\, \leq \, \sqrt{N} \, \|\eta\|\,.$$
\end{proof}
\begin{lemma} \label{Pcom} Assume that $G$ has property $(P_{\rm com})$. \\
Let $F$ be a finite non-empty subset of $G\setminus\{e\}$, $a_g\, \in A$ for each $g \in F$, and set
$y_0 = \sum_{g\, \in \,F} \, \pi(a_g)\, \lambda(g) \in B$.
Then we have
$$ 0 \, \in \, \overline{\text{co}\{v\, y_0 \, v^*\, |Ê\, v \in \U_G\}}^{\,\, \|\cdot\|}\,.$$

\end{lemma}
\begin{proof} Since  $G$ has property $(P_{\rm com})$, we may pick $n\in \Naturali$, $g_0 \in G$ and subsets $U, \, D_1,\, \ldots, \, D_n$ of $G$ so that $(i), (ii)$ and $(iii)$ in the definition of property $(P_{\rm com})$ hold with respect to the given $F$. 

\medskip For each $j \in \Naturali$, we set $g_j = g_0^{\, -j}$. Moreover, for each $N\in \Naturali$, we set $$y_N = \frac{1}{N} \, \sum_{j=1}^N \, \lambda(g_j) \, y_0 \, \lambda(g_j)^*\, \in \, \text{co}\{v\, y_0 \, v^*\, |Ê\, v \in \U_G\}\,.$$ 
We will show that 
\begin{equation}\label{y-ineq}
\| y_N\| \, \leq \, \frac{2\, n}{\sqrt{N}} \, \sum_{g\in F} \, \| a_g\| \, 
\end{equation}

Thus, we will get that $\| y_N\| \to 0 $ as $N \to \infty$, from which the assertion to be proven will clearly follow.

\medskip To prove (\ref{y-ineq}), fix $N\in \Naturali$. Since
$$y_N =  \frac{1}{N} \,\sum_{g\, \in \, F} \sum_{j=1}^N \, \lambda(g_j) \, \pi(a_g)\, \lambda(g) \, \lambda(g_j)^*\,,$$
we have
\begin{equation}\label{z-ineq}\|y_N\| \, \leq \, \frac{1}{N} \,\sum_{g\, \in \, F} \, \|z_g\|\,,
\end{equation} 
where $z_g = \sum_{j=1}^N \, \lambda(g_j) \, \pi(a_g)\, \lambda(g) \, \lambda(g_j)^*$ for each $g \in F$.

\medskip Let $g \in F$ and $\xi, \, \eta \in \ell^2(G, \H)$.  As condition $(iii)$ implies that  for each $k\in \{1, 2, \ldots, n\}$, the sets $g_1D_k,\, \ldots, \, g_N D_k$ are pairwise disjoint,  Lemma \ref{D-ineq} gives that
\begin{equation}\label{DD-ineq}
 \sum_{j=1}^N \, \|P_{\,g_jD_k}\, \eta\| \, \leq \, \sqrt{N} \, \|\eta\|\,  \quad\text{and} \quad   \sum_{j=1}^N \, \|P_{\,g_jD_k}\, \xi\| \, \leq \, \sqrt{N} \, \|\xi\|
\end{equation}
Now, using Lemma \ref{fund} $N$-times (with $M_F = \pi(a_g)$) at the second step, we get
$$|\langle z_g\, \xi\,, \, \eta\rangle| \, \leq \sum_{j=1}^N \, \big|\big\langle \,\pi(a_g) \lambda(g)\, \lambda(g_j)^* \, \xi\,, \, \lambda(g_j)^*\, \eta\big\rangle\big|$$
$$\leq\, \sum_{j=1}^N\sum_{k=1}^n \Big(\|\pi(a_g)\lambda(g)\,\lambda(g_j)^* \, \xi\|\, \|P_{D_k} \lambda(g_j)^*\, \eta\| + \|P_{D_k} \lambda(g_j)^*\, \xi\|\, \| \pi(a_g)^*\, \lambda(g_j)^*\, \eta\|\Big)$$
$$\leq\, \sum_{j=1}^N\sum_{k=1}^n \Big(\|\pi(a_g)\|\, \|\xi\|\, \|P_{g_jD_k} \, \eta\| + \|P_{g_jD_k} \, \xi\|\, \| \pi(a_g)\|\,\|\eta\|\Big)$$
$$=\, \|a_g\|\, \sum_{k=1}^n \Big(\|\xi\|\, \big(\sum_{j=1}^N \|P_{g_jD_k} \, \eta\| \big)\, + \, \|\eta\|\, \big(\sum_{j=1}^N \|P_{g_jD_k} \, \xi\|\big)\Big)$$
$$\leq\,  \|a_g\|\, 2\, n \,  \sqrt{N} \, \|\xi\|\, \|\eta\|\,,$$
where we have used (\ref{DD-ineq}) to get the final inequality.

This implies that $$\|z_g\| \, \leq \, 2\, n \,  \sqrt{N}\, \|a_g\|\,. $$ 
Using (\ref{z-ineq}), we therefore get 
$$\|y_N\| \, \leq \, \frac{1}{N} \, 2\, n \, \sqrt{N} \, \sum_{g\in F} \, \| a_g\|= \frac{2\, n}{\sqrt{N}} \,\sum_{g\in F} \, \| a_g\| \, , $$
that is, the inequality (\ref{y-ineq}) holds, as desired.

\end{proof}
 
 \begin{theorem} \label{PcomDP}
 Assume that $G$ has property $(P_{\rm com})$. Then $\Sigma$ has property $($DP$)$.
 \end{theorem}
 
\begin{proof} Lemma \ref{Pcom}Ê shows that if $x\in B_0$ satisfies $E(x) = 0$, and $\varepsilon > 0$, then there exists a $G$-averaging process on $B$ such that $\|\psi(x)\| < \varepsilon$. Hence, it follows that $\Sigma$ has (the strong) property (DP).

\end{proof}
 
Note that  the proof of Theorem \ref{PcomDP} in fact implies that when $G$ has property $(P_{\rm com})$, then $\Sigma$ satisfies that
\begin{equation}\label{Co3}  0 \, \in \, \overline{\text{co}\{v\, y \, v^*\, |Ê\, v \in \U_G\}}^{\,\, \|\cdot\|}
\end{equation}
for every $y\in B$ satisfying  $ \, E(y) =0$.
As mentioned in Remark \ref{strongDP}, this is true whenever $\Sigma$ satisfies the strong form of property (DP) (hence also when $G$ is a PH group): 

\begin{proposition} \label{strDP} Assume that $\Sigma$ satisfies the strong form of property (DP). Then (\ref{Co3}) holds for every $y\in B$ satisfying  $ \, E(y) =0$.
\end{proposition}
\begin{proof} Let $y\in B$ satisfy  $ \, E(y) =0$ and $\varepsilon >0$.  Write $y = x_1+ i \,x_2$, where $x_1 = {\rm Re}(y)$,\, $x_2={\rm Im}(y)$.
Note that $E(x_1)= \frac{1}{2}\,(E(y) + E(y)^*) = 0$, and, similarly, $E(x_2) = 0$. Using the assumption, we can find a $G$-averaging process $\psi_1$ on $B$ such that $\|\psi_1(x_1)\| < \varepsilon/2$. Now, set $\tilde{x}_2 = \psi_1(x_2)$. Then $\tilde{x}_2$ is self-adjoint, and, using the equivariance property of $E$, one deduces that $E(\tilde{x}_2) = E(x_2) = 0$. Hence, we  can find a $G$-averaging process $\psi_2$ on $B$ such that $\|\psi_2(\tilde{x}_2)\| < \varepsilon/2$. Set $\psi = \psi_2 \circ \psi_1$. Then we get
$$\|\psi(y)\| \leq \|\psi(x_1)\| + \|\psi(x_2)\| \leq \|\psi_1(x_1)\| +  \|\psi_2(\tilde{x}_2)\| < \varepsilon\,,$$
and it follows that (\ref{Co3}) holds.

\end{proof}

\bigskip
{\parindent=0pt Addresses of the authors:\\

\smallskip Erik B\'edos, Institute of Mathematics, University of
Oslo, \\
P.B. 1053 Blindern, N-0316 Oslo, Norway.\\ E-mail: bedos@math.uio.no. \\

\smallskip \noindent
Roberto Conti, Dipartimento di Scienze di Base e Applicate per l'Ingegneria, \\
Sezione di Matematica, 
Sapienza Universit\`a di Roma \\
Via A. Scarpa 16,
I-00161 Roma, Italy.
\\ E-mail: roberto.conti@sbai.uniroma1.it
\par}

\end{document}